\documentclass[12pt]{article}
\usepackage{epsfig}
\usepackage{amsfonts}
\usepackage{amsmath}
\usepackage{amssymb}
\usepackage{amsthm}
\usepackage[pdftex]{hyperref}
\usepackage{verbatim}
\usepackage{multirow}
\usepackage{longtable}
\usepackage{rawfonts}
\usepackage{stmaryrd}
\usepackage{pgf}
\usepackage{tikz}
\usetikzlibrary{automata,calc,shapes.geometric,positioning}
\usepackage[latin1]{inputenc}
\usepackage{mathtools}

\usepackage{float}
\usepackage{setspace}
\usetikzlibrary{arrows.meta, math}

\hypersetup{
  colorlinks,
  urlcolor=black,
  linkcolor=black,
  anchorcolor=black,
  citecolor=black,
}

\allowdisplaybreaks

\renewcommand{\phi}{\varphi}

\newtheorem{theorem}{Theorem}
\numberwithin{theorem}{section}
\newtheorem{conjecture}{Conjecture}

\newtheorem{corollary}[theorem]{Corollary}
\newtheorem{lem}[theorem]{Lemma}
\newtheorem{definition}[theorem]{Definition}

\title{4-Chromatic Graphs Have At Least 4 Cycles of Length 0 mod $3$}
\author{Sean Kim\thanks{Department of Mathematics, California State University San Marcos, San Marcos, CA 92096. E-mail addresses: \href{mailto:kim296@csusm.edu}{\tt kim296@csusm.edu}, \href{mailto:mpicollelli@csusm.edu}{\tt mpicollelli@csusm.edu}}  \and Michael E. Picollelli\footnotemark[1]}
\date{}

\begin{document}

\maketitle

\begin{abstract}
A 2018 conjecture of Brewster, McGuinness, Moore, and Noel \cite{BMMN} asserts that for $k \ge 3$, if a graph has chromatic number greater than $k$, then it contains at least as many cycles of length $0 \bmod k$ as the complete graph on $k+1$ vertices. Our main result confirms this in the $k=3$ case by showing every $4$-critical graph contains at least $4$ cycles of length $0 \bmod 3$, and that $K_4$ is the unique such graph achieving the minimum.

We make progress on the general conjecture as well, showing that $(k+1)$-critical graphs with minimum degree $k$ have at least as many cycles of length $0\bmod r$ as $K_{k+1}$, provided $k+1 \ne 0 \bmod r$.  We also show that $K_{k+1}$ uniquely minimizes the number of cycles of length $1\bmod k$ among all $(k+1)$-critical graphs, strengthening a recent result of Moore and West \cite{MW} and extending it to the $k=3$ case.
\end{abstract}

\section{Introduction}

The study of cycles with a given length modulo an integer $k$ began with work of Burr and Erd\H{o}s \cite{Erd}, who conjectured that for odd values of $k$, graphs of sufficiently large average degree contain cycles of all lengths modulo $k$. This conjecture was proven by Bollob\'as \cite{Bol}, with the bound subsequently improved by Thomassen \cite{Thomassen}.  Further work, building in part on conjectures of Thomassen \cite{Thomassen}, has largely focused on the existence of such cycles under minimum degree or connectivity assumptions -- see, for example,  \cite{Dean}, \cite{Diwan}, \cite{GHLM}, or \cite{SV}.

This paper is concerned with bounding the number of cycles of length $0\bmod k$ and $1\bmod k$ in non-$k$-colorable graphs. The first result in this direction is due to Tuza \cite{Tuza}, who showed that graphs without cycles of length $1 \bmod k$ are $k$-colorable, generalizing K\"onig's classic characterization of bipartite graphs. For cycles of length $0\bmod k$, Chen and Saito \cite{ChenSaito} showed that graphs without cycles of length $0\bmod 3$ are $3$-colorable; Dean, Lesniak and Saito \cite{DLS} achieved the same conclusion for graphs without cycles of length $0\bmod 4$. Chen, Ma and Zang \cite{CMZ} established that any non-$k$-colorable graph is guaranteed cycles of every length $\ell \bmod k$ except possibly $\ell=2$, settling the existence case in the $\ell=0$ case.  Very recent work has essentially settled the existence question entirely: Gao, Huo, and Ma \cite{GHM} showed every $(k+1)$-critical non-complete graph has cycles of all lengths modulo $k$ when $k \ge 6$. The same conclusion for the $k=3$ case was implied by work in \cite{ChenSaito}, \cite{LY}, and \cite{Saito}, and the case where $k \in \{4,5\}$ is the subject of recent work of Huo \cite{Huo}.

The question of how many such cycles must occur has received considerably less attention. Brewster, McGuinness, Moore and Noel \cite{BMMN} generalized a proof of Chen and Saito's result due to Wrochna to show that any graph which is not $k$-colorable has at least $(k-1)!/2$ cycles of length $0\bmod k$. They further conjectured that the complete graph $K_{k+1}$ achieves the minimum number of such cycles among all non-$k$-colorable graphs.

\begin{conjecture}[Brewster, McGuinness, Moore, and Noel, \cite{BMMN}]
\label{conj:BMMN}
If $\chi(G)>k$, then $G$ contains at least $(k+1)(k-1)!/2$ cycles of length $0 \bmod k$.
\end{conjecture}

We will establish this conjecture for $(k+1)$-critical graphs with minimum degree $k$ (see Corollary~\ref{cor:0modkwithmindeg}), but our main contribution is to settle it in the affirmative in the $k=3$ case.

\begin{theorem}\label{thm:mainresult}
If $G$ is $4$-critical, then $G$ has at least $4$ cycles of length $0 \mod 3$, with equality only for $K_4$.
\end{theorem}

Building on the arguments in \cite{BMMN}, Moore and West \cite{MW} established bounds on the number of cycles of length $0 \bmod r$ and $1\bmod r$ in non-$k$-colorable graphs, where $3 \le r \le k$. While their bound on cycles of length $0\bmod k$ was the same as in \cite{BMMN}, they were able to bound the number of cycles of length $1\bmod k$ in non-$k$-colorable graphs below by $k!/2$. This is best possible as it is achieved by $K_{k+1}$. Their argument, a probabilistic application of Tuza's strengthening of Minty's Theorem, also established a structural condition for equality when $k\ge 4$.

\begin{theorem}[Moore and West \cite{MW}, Theorem 5]\label{thm:moorewestthm5}
For $k \ge 3$, a non-$k$-colorable graph has at least $k!/2$ cycles with lengths
congruent to 1 modulo $k$, with equality for $k \ge 4$ only when these cycles all have length
$k + 1$.
\end{theorem}

By modifying their arguments, we achieve a slight strengthening of their result that extends to the $k=3$ case, and relies on a simpler constructive argument.

\begin{theorem}\label{thm:1modkbound}
Let $k \ge 3$ and let $G$ be a $(k+1)$-critical graph. Then $G$ has at least $k!/2$ cycles of length $1 \bmod k$, with equality only for the complete graph $K_{k+1}$.\\
\end{theorem}

The remainder of this paper is organized as follows: in Section~\ref{sec:k+1general} we establish some results for $(k+1)$-critical graphs, followed by the proof of Theorem~\ref{thm:1modkbound} in Section~\ref{sec:1modk}.  Section~\ref{sec:thm1proof} is then focused on the proof of Theorem~\ref{thm:mainresult}.

\section{Cycles in $(k+1)$-critical graphs}
\label{sec:k+1general}

In this section we will establish our results on $(k+1)$-chromatic graphs. We remark that our notation largely follows \cite{WestCM}, and we note that a graph $G$ is {\bf $(k+1)$-critical} if $G$ has chromatic number $k+1$ but every proper subgraph has chromatic number at most $k$.

The results in this section are primarily established by modifying the approach to generalized Kempe chains used by Moore and West in \cite{MW}.  To that end, we will use the the following definition from \cite{MW}:

\begin{definition}[Moore and West \cite{MW}]
    Let $G$ be a $k$-colorable graph and let $\varphi$ be a proper $k$-coloring of $G$, where we assume the set of colors is $[k]=\{1,2,\ldots,k\}$.  For some $r$, $3 \le r \le k$, let $\sigma$ be a cyclic permutation of $r$ distinct elements from $[k]$. Then we define the {\bf $\sigma$-subdigraph $D_{\sigma}$} to be the directed graph satisfying
\begin{enumerate}
    \item $V(D_{\sigma})=V(G)$.
    \item For $u,v \in V(G)$, $uv \in E(D_{\sigma})$ if and only if $uv \in E(G)$ and $\sigma(\varphi(u)) = \varphi(v)$.
\end{enumerate}
\end{definition}

For $v \in V(G)$, let $A_v$ be the set of vertices accessible from $v$ along directed paths in $D_{\sigma}$. A key observation in \cite{MW} is that we may ``shift" the colors in $A_v$ by recoloring each $w \in A_v$ with the color $\sigma(\varphi(w))$; the resulting $k$-coloring will still be a proper. Otherwise, a monochromatic edge $xy$ under the recoloring would have one end $x \in A_v$ and the other end $y \notin A_v$, and must satisfy $\sigma(\varphi(x))=\varphi(y)$, contradicting $y \notin A_v$.

As an illustration, we begin with a slight strengthening of Theorem~4 from \cite{MW} that bounds the number of cycles of a given modular length in a graph $G$ such that $G$ is $(k+1)$-chromatic but $G-e$ is not.  We show further that the bound holds on the number of cycles through each vertex of $e$, although we do not claim that all such cycles found contain both vertices of $e$.  We will use this stronger conclusion later in the proof of Theorem~\ref{thm:mainresult}.

\begin{theorem}[See Moore and West \cite{MW}, Theorem~4]
    \label{thm:my-cycle}
    Fix $r,k \in \mathbb{N}$ with $3 \le r \le k$, and let $e$ be an edge in a graph $G$.  If $\chi(G)=k+1$ but $\chi(G-e)=k$, then letting $x$ be an endpoint of $e$, $G-e$ contains at least $\frac{1}{2} \prod_{i=1}^{r-1} (k-i)$ cycles of length $0\bmod r$.
\end{theorem}

\begin{proof}
    Letting $e = xy$, we fix a proper $k$-coloring $\varphi$ of $G-e$ using colors $1,2,\ldots,k$ and let $\sigma$ be a cyclic $r$-permutation of $[k]$ that includes $\varphi(x)$. We note that $\varphi(x)=\varphi(y)$ must hold, else $\varphi$ is a proper $k$-coloring of $G$.

    Consider the $\sigma$-subdigraph $D_{\sigma}$ of $G$ under $\varphi$ and $\sigma$, and let $A_x$ be the set of vertices accessible from $x$ via directed paths.  By construction of $D_{\sigma}$, the mapping $\varphi'$ formed by setting $\varphi'(v) = \sigma(\varphi(v))$ for $v \in A_x$ and $\varphi'(v)=\varphi(v)$ otherwise is a proper coloring.

    If $y \notin A_x$ then $\varphi'$ yields a proper $k$-coloring of $G$, a contradiction.  Therefore $D_{\sigma}$ contains a directed path $P$ from $x$ to $y$.  By symmetry, it also contains a directed path $Q$ from $y$ to $x$. The concatenation of $P$ and $Q$ yields a closed directed walk in $D_{\sigma}$. It's routine to show that the edges traversed by a closed directed walk can be decomposed into directed cycles, so every vertex on $P$ and $Q$ lies on a directed cycle in $D_{\sigma}$.

    Thus, $D_{\sigma}$ has a directed cycle containing $x$: let $C_{\sigma}$ be its underlying cycle in $G$.    Every directed cycle in $D_{\sigma}$ visits every color from $\sigma$ in order, implying that $C_{\sigma}$ has length $0\bmod r$, and thus the only other $r$-permutation which could produce the same cycle is the inverse $\sigma^{-1}$.  Consequently, the number of cycles found is at least half the number of cyclic $r$-permutations, $\frac{1}{2}\prod_{i=1}^{r-1} (k-i)$.
\end{proof}

As mentioned above, this bound generalizes a result from \cite{BMMN} in the case $r=k$ that establishes that $(k+1)$-chromatic graphs contain at least $(k-1)!/2$ cycles of length $0\bmod k$.  Our goal now is to show the stronger bound of Conjecture~\ref{conj:BMMN} holds for $(k+1)$-critical graphs of minimum degree $k$. We remark that while there exist $(k+1)$-critical graphs of arbitrarily large minimum degree for all $k \ge 3$, we are unaware of any results estimating what proportion of $(k+1)$-critical graphs have minimum degree greater than $k$.  Since the minimum-degree-$k$ case appears to be very `typical' in the literature, and as such graphs are easily produced through the Haj\'os construction, we still think this result is of interest.

We begin with a technical lemma.

\begin{lem}
\label{lem:technical}
    Let $k \ge 3$ and let $G$ be a $(k+1)$-critical graph with $\delta(G)=k$. Let $v$ be a vertex of degree $k$ in $G$ with neighbors $v_1,\ldots,v_k$, and let $\varphi$ be a proper $k$-coloring of $G-v$ satisfying $\varphi(v_i)=i$.  Finally, let $\sigma$ be a cyclic $r$-permutation of elements in $[k]$.  Then the $\sigma$-subdigraph of $G-v$ under $\varphi$ contains directed $v_i,v_{\sigma^{-1}(i)}$-paths $R^i$ for each color $i$ in $\sigma$.
\end{lem}

\begin{proof}
     Fixing $i$ in $\sigma$, we extend $\varphi$ to a proper $k$-coloring $\varphi^i$ of $G-vv_i$ by setting $\varphi^i(v)=i$. Let $D^i_{\sigma}$ denote the $\sigma$-subdigraph of $G-vv_i$ under $\varphi^i$, noting that it contains the $\sigma$-subdigraph of $G-v$ under $\varphi$. We claim that $D_{\sigma}^i$ contains a directed path $Q$ from $v_i$ to $v$. Otherwise, we can shift the colors of the vertices accessible from $v_i$ according to $\sigma$, yielding a proper $k$-coloring of $G$, a contradiction.

     By definition of $D_{\sigma}^i$, the vertex preceding $v$ on $Q$ must be $v_{\sigma^{-1}(i)}$, and therefore $R^i=Q-v$ is a directed $v_i,v_{\sigma^{-1}(i)}$-path lying in the $\sigma$-subdigraph of $G-v$ as claimed.
\end{proof}

Our bound on the number of cycles of length $0 \bmod k$ will follow immediately from the next result.  To state it, we introduce some further notation: for a graph $H$, let $\mathcal{C}(H)$ denote the set of all cycles in $H$, and $\mathcal{C}_{\le i}(H)$ the set of cycles of length at most $i$.

\begin{theorem}
    \label{thm:injectivity}
    Let $k \ge 3$ and let $G$ be a $(k+1)$-critical graph with $\delta(G)=k$.  Then there exists an injective mapping $f:\mathcal{C}_{\le k}(K_{k+1}) \to \mathcal{C}(G)$ such that $\vert V(f(C)) \vert \equiv 0 \bmod \vert V(C) \vert$.
\end{theorem}

\begin{proof}
    We begin by taking $K_{k+1}$ to have vertex set $[k+1]$.  Let $v$ be a vertex in $G$ of degree $k$, and let $v_1,\ldots,v_k$ be its neighbors. Since $G$ is $(k+1)$-critical, let $\varphi$ be a proper $k$-coloring of $G-v$ using colors $1,2,\ldots,k$. Since each neighbor $v_i$ of $v$ must receive a different color under $\varphi$, without loss of generality we assume $\varphi(v_i)=i$.

    We turn now to defining our mapping $f$. Let $C$ be a cycle in $K_{k+1}$ of length $r  \le k$.  Our construction of $f(C)$ differs depending on whether or not $k+1 \in V(C)$.  \\

    Suppose first that $k+1 \notin V(C)$: we orient $C$ into a directed cycle, and let $\sigma$ denote the corresponding cyclic $r$-permutation of $[k]$.  Consider  the $\sigma$-subdigraph of $G-v$ under $\varphi$: by Lemma~\ref{lem:technical}, for each color $i$ in $\sigma$ it contains a directed $v_i,v_{\sigma^{-1}(i)}$-path $R^i$.  The concatenation of these $r$ paths, in the order $(R^i, R^{\sigma^{-1}(i)}, R^{\sigma^{-1}(\sigma^{-1}(i))}, \ldots, R^{\sigma(i)})$, produces a closed directed walk (see Figure~\ref{fig:fruit1}), which must include the edges of at least one directed cycle $\widehat{C}$.  We note that $\widehat{C}$ must follow the colors of $\sigma$ in order, so letting $f(C)$ be its underlying cycle in $G-v$, $f(C)$ has length $0\bmod r$, as required.\\

    Suppose instead that $k+1 \in V(C)$: we construct $f(C)$ to include vertex $v$ as follows.  As above, we orient $C$ into a directed cycle and let $\sigma$ be the corresponding cyclic $r$-permutation, which includes color $k+1$.  Since $r \le k$, let $i \in [k+1]$ be a color {\em not} on $\sigma$ and let $\sigma'$ be the cyclic $r$-permutation of $[k]$ formed by replacing $k+1$ with $i$.

    We next extend $\varphi$ to a proper $k$-coloring $\varphi^i$ of $G-vv_i$ by setting $\varphi^i(v)=i$, and let $D_{\sigma'}$ denote the $\sigma'$-subdigraph of $G-vv_i$ under $\varphi^i$.  By Lemma~\ref{lem:technical}, it contains directed paths $R^{\sigma'(i)}$ and $R^i$ from $v_{\sigma'(i)}$ to $v_i$ and $v_i$ to $v_{\sigma'^{-1}(i)}$, respectively. Thus, the concatenation $(v, R^{\sigma'(i)}, R^i, v)$ is a closed directed walk in $D_{\sigma'}$, and therefore contains the edges of a directed cycle $\widehat{C}$ through vertex $v$. Let $f(C)$ be the underlying cycle of $\widehat{C}$, noting that $f(C)$ has length $0 \bmod r$.\\

    Turning to injectivity, suppose that $f(C^1)=f(C^2)$ for some $C^1, C^2 \in \mathcal{C}^{\le k}(K_{k+1})$, and let $\sigma_1,\sigma_2$ be the cyclic permutations of $C^1,C^2$ used in the construction above.  We observe that $v \in V(f(C))$ if and only if $k+1 \in V(C)$, and we first suppose that $v \notin V(f(C^1))$.  Then under $\varphi$, $f(C^1)$ can be oriented into a directed cycle so that it follows the colors of $\sigma_1$ in order, and it can be oriented so that follows the colors of $\sigma_2$ in order.  But this implies that either $\sigma_1=\sigma_2$, or $\sigma_2$ follows the colors of $\sigma_1$ in the reverse order ($\sigma_2=\sigma_1^{-1}$), and in either case $C^1=C^2$.

    Suppose instead that $v \in f(C^1)$, and let $\sigma_1'$ and $\sigma_2'$ be the cyclic permutations in $[k]$ formed above from $\sigma_1,\sigma_2$ by replacing $k+1$ with some colors $i_1,i_2$, respectively. Then $f(C^1)$ can be oriented into directed cycle(s) $\widehat{C}^1,\widehat{C}^2$ so that by giving vertex $v$ color $i_j$, $\widehat{C}^j$ follows the colors of $\sigma_j'$ in order. If $f(C^1)-v$ has no repeated color under $\varphi$, the directed path $\widehat{C}^j-v$ follows the colors of the permutation $\sigma_j'-\{i_j\}=\sigma_j-\{k+1\}$ in order.  Thus, either $\sigma_2=\sigma_1$ or $\sigma_2=\sigma_1^{-1}$, and $C^1=C^2$ follows.

    If, instead, $f(C^1)-v$ has repeated colors under $\varphi$, then letting $r$ be the least length between repeated entries, it must hold that $\sigma_1',\sigma_2'$ are cyclic $r$-permutations, the first $r-1$ entries of $\widehat{C}^j-v$ form the permutation $\sigma'_j-\{i_j\}$, and  the $r$th vertex of $\widehat{C}^j-v$ is $i_j$.  But this implies $i_1=i_2$, and that either $\sigma_2'=\sigma_1'$ or $\sigma_2'=\sigma_1'^{-1}$, which then implies either $\sigma_2=\sigma_1$ or $\sigma_2=\sigma_1^{-1}$, and $C^1=C^2$ follows, completing the proof that $f$ is injective.
\end{proof}

    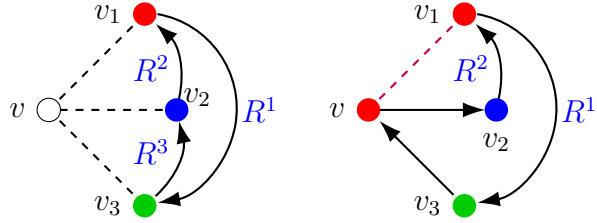
\begin{figure}[]
        \centering

    \tikzmath{\x1 = -4;\x2 = -5;}
   \begin{tikzpicture}[scale=0.85]

            \node (v1) at (\x1,0) {};
            \node (v2) at (\x1+2,0) {};
            \node (v3) at (\x1+1.5,1.5) {};
            \node (v4) at (\x1+1.5,-1.5) {};
            \node (r1) at (\x1+3.3,0) {${\color{blue} R^1}$};
            \node (r2) at (\x1+1.6,.65) {${\color{blue} R^2}$};
            \draw [thick,-{Latex[length=3mm, width=2mm]}] (v1) edge (v2);
            \draw [thick,-{Latex[length=3mm, width=2mm]}] (v4) edge (v1);
            \draw [dashed,thick,purple] (v1) edge (v3);
            \draw [bend right,-{Latex[length=3mm, width=2mm]}, thick] (v2) edge (v3);
            \draw [thick,-{Latex[length=3mm, width=2mm]}] (-2.3,1.5) arc (80:-80:1.5);
            \node [circle, fill=red,scale=0.8,label=left:$v$] at (\x1,0) {};
            \node [circle, fill=blue,scale=0.8,label=below:$v_2$] at (\x1+2,0) {};
            \node [circle, fill=red,scale=0.8,label=left:$v_1$] at (\x1+1.5,1.5) {};
            \node [circle, fill=green!80!black,scale=0.8,label=left:$v_3$] at (\x1 + 1.5,-1.5) {};

            \node (v11) at (\x1+\x2,0) {};
            \node (v12) at (\x1+\x2+2,0) {};
            \node (v13) at (\x1+\x2+1.5,1.5) {};
            \node (v14) at (\x1+\x2+1.5,-1.5) {};
            \node (r1) at (\x1+\x2+3.3,0) {${\color{blue} R^1}$};
            \node (r2) at (\x1+\x2 + 1.6,.65) {${\color{blue} R^2}$};
            \node (r3) at (\x1+\x2 + 1.6,-.65) {${\color{blue} R^3}$};
            \draw [thick,-{Latex[length=3mm, width=2mm]}] (\x1+\x2+1.7,1.5) arc (80:-80:1.5);
            \draw [dashed,thick] (v11) edge (v12);
            \draw [dashed,thick] (v11) edge (v13);
            \draw [dashed,thick] (v11) edge (v14);
            \draw [bend right,-{Latex[length=3mm, width=2mm]}, thick] (v14) edge (v12);
            \draw [bend right,-{Latex[length=3mm, width=2mm]}, thick] (v12) edge (v13);
            \node [circle,draw=black,fill=white,scale=0.8,label=left:$v$] at (\x1+\x2,0) {};
            \node [circle, fill=blue,scale=0.8] at (\x1+\x2+2,0) {};
            \node [circle,label=right:$v_2$] at (\x1+\x2+1.7,.2) {};
            \node [circle, fill=red,scale=0.8,label=left:$v_1$] at (\x1+\x2+1.5,1.5) {};
            \node [circle, fill=green!80!black,scale=0.8,label=left:$v_3$] at (\x1+\x2+1.5,-1.5) {};
        \end{tikzpicture}

        \caption{In the proof of Theorem~\ref{thm:injectivity} in the case $k=3$, we give a view of the closed directed walks used to construct the cycles $f(C)$.  On the left we have the case where $C$ is the triangle on $\{1,2,3\}$ with $\sigma=(1\,2\,3)$, and on the right, the case where $C$ is the triangle on $\{2,3,4\}$ where we take $\sigma=(1\,2\,4)$ and then let $\sigma'=(2\,3\,1)=(1\,2\,3)$.  We remark that despite the depiction, the paths $R^i$ may not be vertex-disjoint.}
        \label{fig:fruit1}
    \end{figure}

    As an immediate consequence we obtain the following corollary.

\begin{corollary}
\label{cor:0modkwithmindeg}
    Let $k \ge 3$ and let $G$ be a $(k+1)$-critical graph with $\delta (G) = k$. Then for each $r$, $2 \le r \le k$, with $k+1 \ne 0 \bmod r$, $G$ contains at least as many cycles of length $0\bmod r$ as the complete graph $K_{k+1}$.  In particular, $G$ has at least $(k+1)(k-1)!/2$ cycles of length $0 \bmod k$.
\end{corollary}

\section{Proof of Theorem~\ref{thm:1modkbound}}
\label{sec:1modk}

Our aim now is to prove Theorem~\ref{thm:1modkbound}: let $k \ge 3$ and let $G$ be a $(k+1)$-critical graph.  Let $v$ be an arbitrary vertex, and let $\varphi$ be a $k$-coloring of $G-v$ using colors $1,2,...,k$. We first show that $G$ contains at least $k!/2$ cycles of length $1 \bmod k$ that include $v$.

For each $i \in [k]$, let $N_i$ be the set of neighbors of $v$ with color $i$ under $\varphi$, noting $N_i \ne \varnothing$ since $G$ is not $k$-colorable. Let $G_i$ be the subgraph of $G$ formed by deleting the edges between $v$ and $N_i$, and let $\varphi^i$ denote the extension of $\varphi$ to a proper $k$-coloring of $G_i$ by setting $\varphi^i(v)=i$.

Let $\sigma$ be a cyclic permutation of $[k]$, and consider the $\sigma$-subdigraph $D^i_{\sigma}$ of $G_i$ under $\varphi^i$. We claim that $D^i_{\sigma}$ must contain a directed path from $v$ into $N_i$.  Otherwise, letting $W$ denote the set of vertices in $D^i_{\sigma}$ accessible from $v$, we can recolor $G_i$ by shifting the colors in $W$ according to $\sigma$, producing a proper $k$-coloring $\varphi'$ of $G_i$ in which all vertices in $N_i$ have color $i$ and $v$ has color $\sigma(i)$. But in that case $\varphi'$ is a proper coloring of $G$, a contradiction.

With foresight we let $v_i \in N_i$ be a vertex on a shortest directed path $P$ in $D_{\sigma}^i$ from $v$ to $N_i$, noting $P$ has length $0 \bmod k$ since it begins and ends on a vertex of color $i$. Then $P$ along with the edge $v_iv$ forms a directed cycle $\widehat{C}_{\sigma}^i$ in $G$ of length $1 \bmod k$; let $C^i_{\sigma}$ be its underlying cycle in $G$.  Since $P$ visits all the colors of $\sigma$ in order and contains the edge $v_iv$, $\widehat{C}_{\sigma}^i$, viewed now as an oriented subgraph of $G$, determines $i$ and $\sigma$. Consequently, $C_{\sigma}^i$ could also be the underlying cycle of $\widehat{C}_{\sigma^{-1}}^{\sigma(i)}$, but of no other such directed cycle constructed this way. So as $i$ and $\sigma$ vary, this construction produces at least $k\cdot (k-1)!/2 = k!/2$ distinct cycles of length $1 \bmod k$ containing $v$.

Suppose now that $G$ has exactly $k!/2$ cycles of length $1 \bmod k$ in total. Our argument shows every vertex lies on at least $k!/2$ such cycles, so every such cycle is spanning. Fixing a cyclic permutation $\sigma$ of $[k]$ and an $i \in [k]$, the constructed cycle $C^i_{\sigma}$ is spanning, and therefore so is the shortest directed path from $v$ to $N_i$ in $D_{\sigma}^i$. Consequently, $|N_i|=1$, implying that $v$ has degree $k$ and thus $G$ is $k$-regular.  Brooks' Theorem then implies $G=K_{k+1}$, completing the proof. \\

\section{Proof of Theorem~\ref{thm:mainresult}}
\label{sec:thm1proof}

This section is focused on providing the proof of Theorem~\ref{thm:mainresult}.  Our argument consists of the following three claims, noting that any 4-critical graph has minimum degree at least 3.
\begin{enumerate}
    \item A 4-critical graph $G$ with $\delta(G) =3$ has at least 4 cycles of length $0\bmod 3$.
    \item $K_4$ is the only 4-critical graph with $\delta(G)=3$ and exactly 4 cycles of length $0\bmod 3$.
    \item A 4-critical graph $G$ with $\delta(G)\ge 4$ has at least 5 cycles of length $0\bmod 3$.
\end{enumerate}

The first claim follows from Corollary~\ref{cor:0modkwithmindeg}. The second claim is the subject of Lemma~\ref{lem:kay4}  below.  The third claim, which is the most technical argument of this paper, follows from Lemmas~\ref{lem:4gooda} and \ref{lem:4goodb} in Section~\ref{sec:MinDeg4} below.

\begin{lem}
\label{lem:kay4}
    $K_4$ is the only 4-critical graph with minimum degree 3 that has exactly 4 cycles of length $0 \bmod 3$.
\end{lem}

\begin{proof}
    Suppose $G$ is 4-critical, $\delta(G)=3$, and $G$ has exactly 4 cycles of length $0 \bmod 3$. Let $v$ be a vertex in $G$ with degree 3 and label its neighbors $v_1,v_2,v_3$.  Let $\varphi$ be a proper $3$-coloring of $G-v$ satisfying $\varphi(v_i)=i$, and let $\sigma=(1\,2\,3)$. Finally, we let $R^1,R^2,R^3$ denote the three directed paths in the $\sigma$-subdigraph of $G-v$ guaranteed to exist by Lemma~\ref{lem:technical}.

    For $i \in [3]$, let $\varphi^i$ be the extension of $\varphi$ to $G-vv_i$ formed by setting $\varphi^i(v)=i$, and let $D_{\sigma}^i$ be the $\sigma$-subdigraph of $G-vv_i$ under $\varphi^i$.  Then the closed walk in $D_{\sigma}^i$ formed by the concatenation $(v,R^{\sigma(i)},R^i,v)$ includes the (oriented) edges of a cycle $C^i$ in $G$ containing edges $vv_i$ and $vv_{\sigma^{-1}(i)}$ in $G$ of length $0\bmod 3$, and the closed walk formed by the concatenation $(R_1,R_3,R_2)$ includes the edges of a cycle $C^4$ in $G-v$ of length $0\bmod 3$.  (The closed walks are also illustrated in Figure~\ref{fig:fruit1}.)

    As the cycles $C^1,C^2,C^3,C^4$ are distinct, these are the only cycles of length $0\bmod 3$ in $G$. It follows then that the closed walk $(R_1,R_3,R_2)$ includes only the edges of $C^4$, and, by construction, $C^4$ contains the paths $C^i-v$, $i \in [3]$.  We further claim that $C^4$ spans $G-v$, as any vertex $x$ in $G-v$ not lying on $C^4$ must lie on a (fifth) cycle of length $0\bmod 3$ by Theorem~\ref{thm:my-cycle}, taking $e$ to be any edge incident to $x$, a contradiction.

    If $C^4$ is a triangle then it follows that $G=K_4$, so we suppose otherwise.  Since $\sigma=(1\,2\,3)$, we can label $C^4$'s vertices $x_1,x_2,x_3,\ldots,x_{3k}$ so that $E(C^4)=\{x_1x_2,x_2x_3,\ldots,x_{3k-1}x_{3k},x_{3k}x_1\}$ and
    \[\varphi(x_i) = \begin{cases}
    1 & \mbox{ if } i \equiv 1\bmod 3,\\
    2 & \mbox{ if } i \equiv 2\bmod 3,\\
    3 & \mbox{ if } i \equiv 0\bmod 3.
    \end{cases}\]
    Since $G$ has minimum degree 3 and $v$ has only three neighbors, $C^4$ contains chords. But chords can only connect vertices of different colors, and any such chord cuts $C^4$ into two cycles, one of which has length $0 \bmod 3$, a contradiction that completes the proof. (See Figure~\ref{fig:fruit2}.)

    \begin{figure}
        \centering
        \begin{tikzpicture}[scale=0.45]

            \node (v1) at (4,0) {};
            \node (v2) at (3.06,2.57) {};
            \node (v3) at (0.69,3.94) {};
            \node (v13) at (.19,4.94) {};
            \node (v4) at (-2,3.46) {};
            \node (v5) at (-3.76,1.37) {};
            \node (v6) at (-3.76,-1.37) {};
            \node (v7) at (-2,-3.46) {};
            \node (v8) at (0.69,-3.94) {};
            \node (v18) at (.29,-5.13) {};
            \node (v9) at (3.06,-2.57) {};
            \draw [-{Latex[length=3mm, width=2mm]}, thick] (v1) edge (v2);
            \draw [-{Latex[length=3mm, width=2mm]}, thick] (v2) edge (v3);
            \draw [-{Latex[length=3mm, width=2mm]}, thick] (v3) edge (v4);
            \draw [-{Latex[length=3mm, width=2mm]}, thick] (v4) edge (v5);
            \draw [-{Latex[length=3mm, width=2mm]}, thick] (v5) edge (v6);
            \draw [-{Latex[length=3mm, width=2mm]}, thick] (v6) edge (v7);
            \draw [-{Latex[length=3mm, width=2mm]}, thick] (v7) edge (v8);
            \draw [-{Latex[length=3mm, width=2mm]}, thick] (v8) edge (v9);
            \draw [-{Latex[length=3mm, width=2mm]}, thick] (v9) edge (v1);
            \draw [-{Latex[length=3mm, width=2mm]},dashed,thick] (v8) edge (v3);
            \draw [red,thick] (-1.2,0) ellipse (3.2 and 6);
            \node [circle,fill=red,scale=0.8,label=right:$x_1$] at (4,0) {};
            \node [circle,fill=blue,scale=0.8,label=above:$x_2$] at (3.06,2.57) {};
            \node [circle,fill=green!80!black,scale=0.8,label=above:$x_3$] at (0.69,3.94) {};
            \node [circle,fill=red,scale=0.8,label=above:$x_4$] at (-2,3.46) {};
            \node [circle,fill=blue,scale=0.8,label=above:$x_5$] at (-3.76,1.37) {};
            \node [circle,fill=green!80!black,scale=0.8,label=left:$x_6$] at (-3.76,-1.37) {};
            \node [circle,fill=red,scale=0.8,label=above:$x_7$] at (-2,-3.46) {};
            \node [circle,fill=blue,scale=0.8,label=above:$x_8$] at (0.69,-3.94) {};
            \node [circle,fill=green!80!black,scale=0.8,label=right:$x_9$] at (3.06,-2.57) {};

        \end{tikzpicture}
        \caption{In the proof of Lemma~\ref{lem:kay4}, an illustration of the cycle $C^4$ containing a chord, which necessarily produces a fifth cycle of length $0 \bmod 3$.}
        \label{fig:fruit2}
    \end{figure}
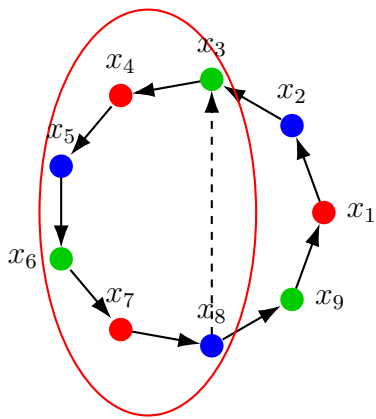
\end{proof}

\subsection{4-critical graphs with minimum degree at least 4}
\label{sec:MinDeg4}

To finish the proof of Theorem~\ref{thm:mainresult}, we must now show that every 4-critical graph $G$ with minimum degree at least 4 contains at least 5 cycles of length $0\bmod 3$. In this case, the argument used in Theorem~\ref{thm:injectivity} does not directly adapt as it obtained cycles from closed directed walks. A key to constructing these walks was that every neighbor of a vertex $v$ of degree 3 in a 4-critical graph receives a different color under a proper 3-coloring of $G-v$, which does not apply once $\delta(G) \ge 4$.

Instead, our arguments will rely on two other results.  The first is a sufficient condition for cycles of length $0 \bmod 3$ to exist due to Chen and Saito \cite{ChenSaito}:

\begin{theorem}[Chen and Saito \cite{ChenSaito}, Theorem 1]
    \label{thm:ChenSaito}
    If $G$ is a graph with $n \ge 2$ vertices and at most one vertex of degree 2 or less, then $G$ contains a cycle of length $0 \bmod 3$.
\end{theorem}

We remark that in \cite{ChenSaito}, the authors refer to cycles with length divisible by $3$ as {\bf good cycles}, and we will do so for the remainder of this section. Our first application of this will be to handle the case where a vertex lies on at least three good cycles.

\begin{lem}
\label{lem:4gooda}
    If $G$ is a 4-critical graph with $\delta(G)\ge 4$ and there exists a vertex $v \in V(G)$ such that $v$ lies on at least 3 good cycles, then $G$ contains at least 5 good cycles.
\end{lem}

\begin{proof}
    Suppose $v \in V(G)$ lies on at least 3 good cycles, call them $C^1,C^2,C^3$.  Since $\delta(G) \ge 4$, $\delta(G-v) \ge 3$, so $G-v$ contains at least one good cycle $C^4$ by Theorem~\ref{thm:ChenSaito}.

   If $V(C^4) \nsubseteq N(v)$ then $C^4$ contains an edge $e$ with at most one endpoint in $N(v)$. Deleting $e$ from $G-v$ produces a subgraph with at most one vertex of degree 2, and thus it contains a good cycle $C^5\ne C^4$ by Theorem~\ref{thm:ChenSaito}, yielding the result.

   If $V(C^4) \subseteq N(v)$, then every edge of $C^4$ forms a triangle with $v$, and we're done if $|V(C^4)|\ge 5$.  But if $|V(C^4)|<5$, then $|V(C^4)|=3$ since $C^4$ is good. This implies $G$ contains a $K_4$ on $C^4 \cup \{v\}$, contradicting $G$'s 4-criticality and completing the proof.

\end{proof}

We'll also appeal to a recent result of Gao, Huo, Liu and Ma \cite{GHLM}, which was used to resolve a number of conjectures regarding the existence of cycles of prescribed lengths.

\begin{definition}[Gao, Huo, Liu and Ma \cite{GHLM}]
    A collection of $\ell$ paths is \textbf{admissible} if the length of every path is at least two, and the lengths form an arithmetic progression with common difference one or two.
\end{definition}

\begin{theorem} [Gao, Huo, Liu, and Ma \cite{GHLM}, Theorem 1.2 ]
\label{thm:admissible}
    Let $G$ be a 2-connected graph and let $x,y$ be distinct vertices of $G$. If every vertex in $G$ other than $x$ and $y$ has degree at least $\ell+1$, then there exist $\ell$ admissible paths from $x$ to $y$ in $G$.\\
\end{theorem}

We note that Theorem~\ref{thm:my-cycle} implies that every vertex $v$ in a 4-critical graph lies on at least two good cycles, by applying the Theorem first with an arbitrary edge incident with $v$, then with an edge incident with $v$ on the found cycle.

\begin{lem}
\label{lem:exactlyone}
    If $G$ is $4$-critical with $\delta(G) \ge 4$ and every vertex is on at most two good cycles, then:
    \begin{enumerate}
        \item $G$ is $4$-regular, and
        \item Every edge of $G$ lies on exactly one good cycle.
\end{enumerate}
\end{lem}

\begin{proof}
It suffices to argue that every edge lies on at least one good cycle.  This will imply that every vertex lies on at least $d(v)/2$ good cycles, yielding a maximum degree of at most 4 under the assumptions given and therefore $G$ is $4$-regular.  Furthermore, if any edge $uv$ lies on at least two good cycles, then those two cycles only cover three edges incident with $v$, implying $v$ lies on a third good cycle, a contradiction.

To that end, pick any edge $e=uv$. It is routine to show that a color-critical graph is $2$-connected (e.g., see \cite{WestCM}, Exercise~8.2.11), so by Theorem~\ref{thm:admissible} there exist $3$ admissible paths from $u$ to $v$. Since the common difference is one or two, their lengths cover the congruence classes modulo 3, so one of these paths has length $2\bmod 3$ and forms a good cycle with $e$.
\end{proof}

The next lemma will complete the proof of Theorem~\ref{thm:mainresult}.

\begin{lem}
\label{lem:4goodb}
    If $G$ is a 4-critical graph with minimum degree at least 4 and every vertex is on at most two good cycles, then $G$ contains at least 5 good cycles.
\end{lem}

\begin{proof}
    We first observe that under the given assumptions and by Theorem~\ref{thm:my-cycle}, every vertex lies on exactly two good cycles. Furthermore, deleting any vertex yields a subgraph with minimum degree at least 3, which contains a good cycle by Theorem~\ref{thm:ChenSaito}. Thus, $G$ has at least three good cycles, at least one of which is not spanning.

    Let $C^1$ be a non-spanning good cycle, and let $uv$ be an edge connecting $C^1$ to the rest of $G$, where $u \in V(C^1)$ and $v \in V(G)-V(C^1)$.  Then, by Lemma~\ref{lem:exactlyone}, the edge $uv$ lies on a good cycle $C^2$, and then $v$ must lie on a second good cycle $C^3$.  Our argument proceeds by considering two cases: whether or not $C^2$ is a triangle.\\

    \noindent {\bf Case 1: $C^2$ is not a triangle.} By Lemma~\ref{lem:exactlyone}, $C^2$ is the only good cycle containing edge $uv$, and $G$ is $4$-regular. Consequently, $N_G(u) \cap N_G(v) = \emptyset$, so let $U = N_G(u)-v$ and $V = N_G(v)-u$, and label their vertices $u_1, u_2, u_3$ and $v_1, v_2, v_3$, respectively. Observe that if $H = G - u - v$, then for any vertex $w \in U \cup V$, $d_H(w) = 3$, while for any vertex $ x \in (V(H) - U \cup V)$, $d_H(x) = 4$. Thus, by Theorem~\ref{thm:ChenSaito}, there exists a fourth good cycle that does not include $u$ nor $v$; label it $C$. We have 2 subcases for $C$:

    \begin{enumerate}
        \item $C$ has an edge $e$ with an endpoint not in $U\cup V$
        \item $V(C) \subseteq U \cup V$.
    \end{enumerate}

    In Subcase 1, the edge $e$ has at most one endpoint of degree 3. So, its removal will result in a graph with at most one vertex of degree 2. Therefore, by Theorem~\ref{thm:ChenSaito}, we can see that there must exist another good cycle, call it $C'$, that does not contain $u$, $v$, nor $e$. Thus, $C'$ is our fifth good cycle in $G$.

    For Subcase 2, note that no edge in $C$ is contained in $U$ or in $V$, else it lies on a triangle, contradicting the fact that every edge is on exactly one good cycle. Therefore, $C$ is bipartite, and because $C$ is good, it must be a 6-cycle that spans $U \cup V$.

    Without loss of generality, we may assume $C = [u_1,v_1,u_2,v_2,u_3,v_3]$. Now, let $C'=[u,u_1,v_1,u_2,v_2,v]$. $C$ and $C'$ are two different good cycles that share common edges; a contradiction. (See Figure~\ref{fig:fruit8}.) Thus, in Case 1, $G$ has at least 5 good cycles.

    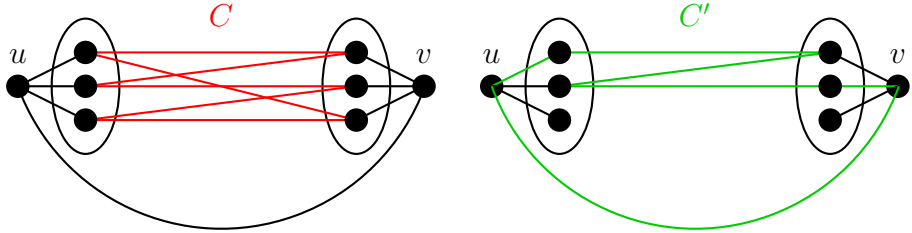
\begin{figure}[H]
        \centering
        \begin{tikzpicture}[scale=.45]
            \draw [thick] (-4,0) ellipse (1 and 2);
            \draw [thick] (4,0) ellipse (1 and 2);
            \node [circle, fill=black, scale=.8, label=above:$u$] at (-6,0) {};
            \node [circle, fill=black, scale=.8, label=above:$v$] at (6,0) {};
            \draw [thick] (-6,0) edge (-4,1);
            \draw [thick] (-6,0) edge (-4,0);
            \draw [thick] (-6,0) edge (-4,-1);
            \draw [thick] (6,0) edge (4,1);
            \draw [thick] (6,0) edge (4,0);
            \draw [thick] (6,0) edge (4,-1);
            \draw [thick,red] (-4,1) edge (4,1);
            \draw [thick,red] (4,1) edge (-4,0);
            \draw [thick,red] (-4,0) edge (4,0);
            \draw [thick,red] (4,0) edge (-4,-1);
            \draw [thick,red] (-4,-1) edge (4,-1);
            \draw [thick,red] (4,-1) edge (-4,1);
            \draw [thick] (-6,0) arc (200:340:6.4);
            \node [circle, fill=black, scale=.8] at (-4,1) {};
            \node [circle, fill=black, scale=.8] at (-4,0) {};
            \node [circle, fill=black, scale=.8] at (-4,-1) {};
            \node [circle, fill=black, scale=.8] at (4,1) {};
            \node [circle, fill=black, scale=.8] at (4,0) {};
            \node [circle, fill=black, scale=.8] at (4,-1) {};

            \draw [thick] (10,0) ellipse (1 and 2);
            \draw [thick] (18,0) ellipse (1 and 2);
            \node [circle, fill=black, scale=.8, label=above:$u$] at (8,0) {};
            \node [circle, fill=black, scale=.8, label=above:$v$] at (20,0) {};
            \draw [thick,green!80!black] (8,0) edge (10,1);
            \draw [thick] (8,0) edge (10,0);
            \draw [thick] (8,0) edge (10,-1);
            \draw [thick] (20,0) edge (18,1);
            \draw [thick,green!80!black] (20,0) edge (18,0);
            \draw [thick] (20,0) edge (18,-1);
            \draw [thick,green!80!black] (10,1) edge (18,1);
            \draw [thick,green!80!black] (18,1) edge (10,0);
            \draw [thick,green!80!black] (10,0) edge (18,0);
            \draw [thick,green!80!black] (8,0) arc (200:340:6.4);
            \node [circle, label=above:\textcolor{red}{$C$}] at (0,1) {};
            \node [circle, label=above:\textcolor{green!80!black}{$C'$}] at (14,1) {};
            \node [circle, fill=black, scale=.8] at (10,1) {};
            \node [circle, fill=black, scale=.8] at (10,0) {};
            \node [circle, fill=black, scale=.8] at (10,-1) {};
            \node [circle, fill=black, scale=.8] at (18,1) {};
            \node [circle, fill=black, scale=.8] at (18,0) {};
            \node [circle, fill=black, scale=.8] at (18,-1) {};
        \end{tikzpicture}
        \caption{This image represents the fourth and fifth good cycles constructed in Subcase 2 of Case 1, noting this yields a contradiction as they share edges.
        }
        \label{fig:fruit8}
    \end{figure}

    \noindent {\bf Case 2: $C^2$ is a triangle.} Since $C^2$ is a triangle, there is a vertex $w$ such that $V(C^2) = \{u,v,w\}$. Let $U = N_G(u) - v - w$, $V = N_G(v) - u - w$, and $W = N_G(w) - u - v$. We note that $U,V,W$ are pairwise disjoint, else we contradict Lemma~\ref{lem:exactlyone}, and we may label their vertices $U=\{u_1,u_2\}, V = \{v_1,v_2\}$, and $W=\{w_1,w_2\}$. In the subgraph $H=G-u-v-w$, all vertices in $U$, $V$, and $W$ will have degree 3, while all other vertices in $H$ will have degree 4. So by Theorem~\ref{thm:ChenSaito}, $H$ has a good cycle $C$ which does not use vertices $u$ or $v$. Thus, $C$ is a fourth good cycle in $G$, and we consider 2 subcases:

    \begin{enumerate}
        \item $C$ has an edge $e$ that has an endpoint not in $U \cup V \cup W$,
        \item $V(C)$ is a subset of $U \cup V \cup W$.
    \end{enumerate}

    In Subcase 1, similar to our argument in Case 1, removing $e$ will result in a graph with at most one vertex with degree at most 2. Then Theorem~\ref{thm:ChenSaito} yields that $G$ contains a fifth distinct good cycle.

    So we will focus on Subcase 2. By similar reasoning as in Subcase 2 of Case 1, we know that $C$ cannot have an edge between $u_1$ and $u_2$, between $v_1$ and $v_2$, or between $w_1$ and $w_2$. So, $C$ is tripartite.  Suppose first that $C$ contains a vertex whose neighbors lie in different parts: without loss of generality, we may suppose that $C$ contains $u_1,v_1,w_1$ and $u_1$ is adjacent to $v_1,w_1$.  Then the cycle $C' =[u_1,v_1,v,u,w,w_1]$ is a different good cycle that shares edges with $C$; a contradiction. (See Figure~\ref{fig:fruit9}.)

    \begin{figure}[H]
        \centering
        \begin{tikzpicture}[scale=.25]
            \draw[thick,color=black] (-10,3.5) ellipse (4 and 1.5);
            \draw[thick,color=black,rotate=60] (-5,4.75) ellipse (4 and 1.5);
            \draw[thick,color=black,rotate=120] (5,12.5) ellipse (4 and 1.5);
            \draw[thick] (-10,8) edge (-8,3.5);
            \draw[thick] (-10,8) edge (-12,3.5);
            \draw[thick] (-17.25,-3.75) edge (-12.35,-3.5);
            \draw[thick] (-17.25,-3.75) edge (-14.5,0);
            \draw[thick] (-2.6,-3.5) edge (-7.65,-3.5);
            \draw[thick] (-2.6,-3.5) edge (-6,-.5);
            \draw[thick,red] (-8,3.5) edge (-14.5,0);
            \draw[thick,red] (-8,3.5) edge (-7.65,-3.5);
            \draw[thick,red] (-7.65,-3.5) edge (-14.5,0);
            \node[circle, fill=black,scale = 0.8] at (-8,3.5) {};
            \node[circle, fill=black,scale = 0.8] at (-12,3.5) {};
            \node[circle, fill=black,scale = 0.8] at (-12.35,-3.5) {};
            \node[circle, fill=black,scale = 0.8] at (-14.5,0) {};
            \node[circle, fill=black,scale = 0.8] at (-7.65,-3.5) {};
            \node[circle, fill=black,scale = 0.8] at (-6,-.5) {};
            \node[circle, fill=black,scale=0.8,label=above:$u$] at (-10,8) {};
            \node[circle, fill=black,scale=0.8,label=left:$v$] at (-17.25,-3.75) {};
            \node[circle, fill=black,scale=0.8,label=right:$w$] at (-2.6,-3.5) {};
            \draw[bend left,thick] (-17.25,-3.75) edge (-10,8);
            \draw[bend left,thick] (-10,8) edge (-2.6,-3.5);
            \draw[bend left,thick] (-2.6,-3.5) edge (-17.25,-3.75);

            \draw[thick,color=black] (10,3.5) ellipse (4 and 1.5);
            \draw[thick,color=black,rotate=60] (5,-12.5) ellipse (4 and 1.5);
            \draw[thick,color=black,rotate=120] (-5,-4.75) ellipse (4 and 1.5);
            \draw[thick] (10,8) edge (8,3.5);
            \draw[thick] (10,8) edge (12,3.5);
            \draw[line width = 2.5,green!80!black] (17.25,-3.75) edge (12.35,-3.5);
            \draw[thick] (17.25,-3.75) edge (14.5,0);
            \draw[thick] (2.6,-3.5) edge (7.65,-3.5);
            \draw[line width = 2.5,green!80!black] (2.6,-3.5) edge (6,-.5);
            \draw[line width = 2.5,green!80!black] (12,3.5) edge (12.35,-3.5);
            \draw[line width = 2.5,green!80!black] (12,3.5) edge (6,-.5);
            \draw[bend right,line width = 2.5,green!80!black] (17.25,-3.75) edge (10,8);
            \draw[bend right,line width = 2.5,green!80!black] (10,8) edge (2.6,-3.5);
            \draw[bend right,thick] (2.6,-3.5) edge (17.25,-3.75);
            \node[circle, fill=black,scale = 0.8] at (8,3.5) {};
            \node[circle, fill=black,scale = 0.8] at (12,3.5) {};
            \node[circle, fill=black,scale = 0.8] at (12.35,-3.5) {};
            \node[circle, fill=black,scale = 0.8] at (14.5,0) {};
            \node[circle, fill=black,scale = 0.8] at (7.65,-3.5) {};
            \node[circle, fill=black,scale = 0.8] at (6,-.5) {};
            \node[circle, fill=black,scale=0.8,label=above:$u$] at (10,8) {};
            \node[circle, fill=black,scale=0.8,label=right:$w$] at (17.25,-3.75) {};
            \node[circle, fill=black,scale=0.8,label=left:$v$] at (2.6,-3.5) {};
        \end{tikzpicture}
        \caption{In Case 2, Subcase 2, the image on the left represents $G$ with a good cycle $C$ containing a vertex with neighbors in both other parts. The image on the right illustrates a second good cycle $C'$ outlined in green, sharing  edges with $C$.}
        \label{fig:fruit9}
    \end{figure}
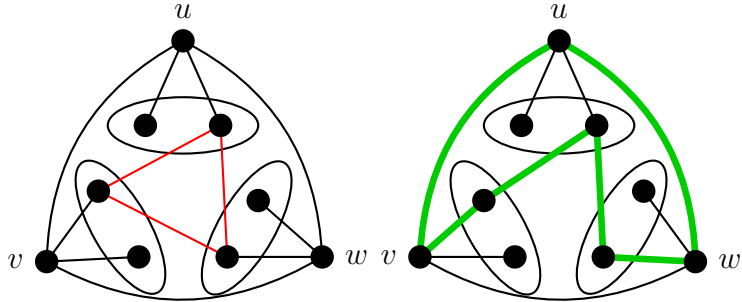

    If no vertex of $C$ has neighbors in different parts, then letting $x \in V(C)$, $x$'s neighbors lie in one part, and its neighbors' neighbors lie in the same part as $x$. That is, $C$ must be contained in two parts and is therefore bipartite on at most 4 vertices, contradicting that $C$ is good.  Thus, in Case 2, $G$ has at least 5 good cycles, completing the proof.

\end{proof}

\subsection*{Acknowledgement}

The authors wish to thank Andr\'e K\"undgen for a careful reading of an early draft and many helpful comments.


\begin{thebibliography}{9}


\bibitem{Bol} B. Bollob\'as, {\em Cycles modulo $k$}, Bull. London Math. Soc. 9 (1977), no. 1, 97-98.



\bibitem{BMMN}
R. C. Brewster, S. McGuinness, B. Moore, and J. A. Noel, \emph{A Dichotomy Theorem for Circular Colouring Reconfiguration}. {Theoret. Comput. Sci.} 639 (2016), 1-13.

\bibitem{ChenSaito}
G.T. Chen, A. Saito, \emph{Graphs with a Cycle of Length Divisible by Three}. {Journal of Combin. Theory, Ser. B}, 60 (1994), 277-292.

\bibitem{CMZ} Z. Chen, J. Ma, and W. Zang, {\em Colouring digraphs with forbidden cycles}, J. Combin. Theory
Ser. B 115 (2015), 210-223.


\bibitem{Dean}
N. Dean, \emph{Which graphs are pancyclic modulo k?}, Sixth Internat. Conf. on the Theory of Applications of Graphs, Kalamazoo, Michigan (1988) 315-326.


\bibitem{DLS}
 N. Dean, L. Lesniak and A. Saito, \emph{Cycles of length 0 modulo 4 in graphs}, {Discrete Math.} 121 (1993), 37-49.

\bibitem{Diwan} A. Diwan, {\em Cycles of even lengths modulo $k$}, J. Graph Theory 65 (2010), no. 3, 246-252.


\bibitem{Erd} P. Erd\H{o}s, {\em Some recent problems and results in graph theory, combinatorics and number theory}, Proceedings of the Seventh Southeastern Conference on Combinatorics, Graph Theory, and Computing (Louisiana State Univ., Baton Rouge, La., 1976), Congress. Numer. XVII , pp. 3--14,

\bibitem{GHLM}
J. Gao, Q. Huo, C. Liu, and J. Ma, \emph{A Unified Proof of Conjectures on Cycle Lengths in Graphs}, {International Mathematics Research Notices} 10 (2021), 7615-7653.

\bibitem{GHM} J. Gao, Q. Huo, and J. Ma, {\em A strengthening on odd cycles in graphs of given chromatic
number}, SIAM J. Discrete Math. 35 (2021), 2317-2327.

\bibitem{Huo} Q. Huo, A note on cycle lengths in graphs of chromatic number five and six, {\tt arXiv:2104.01382}.

\bibitem{MW}
B. Moore and D. B. West, \emph{Cycles in Color-Critical Graphs}. {Electron. J. Comb.} 28 (2021).

\bibitem{LY} M. Lu and Z. Yu, {\em Cycles of length 1 modulo 3 in graph}, Discrete Applied Mathematics 113 (2001),
329-336.

\bibitem{Saito} A. Saito, {\em Cycles of length 2 modulo 3 in graphs}, Discrete Mathematics 101 (1992), 285-289.


\bibitem{SV}
B. Sudakov and J. Verstra\"{e}te, \emph{The extremal function for cycles of length $\ell\bmod k$}, {Electron. J. Combin.} 24(1) (2017), \#P1.7

\bibitem{Thomassen}
C. Thomassen, \emph{Graph decomposition with applications to subdivisions and path systems modulo k}, {J. Graph Theory} 7 (1983), 261-271.

\bibitem{Tuza}
Zs. Tuza, \emph{Graph coloring in linear time.}, {J. Combin. Theory Ser. B} 55 (1992), no. 2, 236-243.

\bibitem{WestCM}
D. West, \emph{Combinatorial Mathematics}. Cambridge University Press, Cambridge, 2021.

\end{thebibliography}
\end{document}